\documentclass[12pt]{amsart}

\usepackage[a4paper,margin=1in]{geometry}
\usepackage{amsmath,amssymb,amsthm,mathtools}
\usepackage{enumitem}
\usepackage[hidelinks]{hyperref}
\usepackage{microtype}
\allowdisplaybreaks

\newcommand{\C}{\mathbb C}
\newcommand{\R}{\mathbb R}
\newcommand{\Pj}{\mathbb P}
\newcommand{\dbar}{\bar\partial}
\newcommand{\Span}{\operatorname{Span}}
\newcommand{\Sym}{\operatorname{Sym}}

\theoremstyle{plain}
\newtheorem{theorem}{Theorem}[section]
\newtheorem{proposition}[theorem]{Proposition}
\newtheorem{lemma}[theorem]{Lemma}

\newtheorem{question}{Question}

\theoremstyle{definition}
\newtheorem{definition}[theorem]{Definition}

\theoremstyle{remark}
\newtheorem{remark}[theorem]{Remark}

\title[Wronskian curvature positivity and holomorphic projective connections]{Wronskian curvature positivity is equivalent to 
the existence of a holomorphic projective connection}

\author{Pengchao Wang}
\address{%
  Academy of Mathematics and Systems Science,
  Chinese Academy of Sciences,
  Beijing 100190, China
}
\email{blowa8@gmail.com}

\author{Song-Yan Xie}
\address{%
  State Key Laboratory of Mathematical Sciences,
  Academy of Mathematics and Systems Science,
  Chinese Academy of Sciences, Beijing 100190, China;\\
  School of Mathematical Sciences,
  University of Chinese Academy of Sciences,
  Beijing 100049, China
}
\email{xiesongyan@amss.ac.cn}



\date{\today}

\begin{document}

\begin{abstract}
Noguchi introduced the notion of Wronskian curvature positivity in his Second Main
Theorem and asked for further examples to which his theorem applies.  We give a
complete answer: a complex manifold admits a smooth connection satisfying this
condition if and only if it admits a holomorphic projective connection.  For a
fixed torsion-free connection, the condition is equivalent to the holomorphicity
of its projective class.
\end{abstract}

\maketitle

\section{Introduction}

\subsection{Noguchi's theorem and the question studied here}
Let $M$ be a compact complex manifold of dimension~$n$, and write
$T_M\coloneqq T^{1,0}M$ for its holomorphic tangent bundle.  Following Noguchi~\cite{Noguchi2011},
by a $C^\infty$ connection in $T_M$ we mean a map
\begin{equation}\label{Noguchi connection}
 (U,V)\longmapsto\nabla_UV,\qquad U,V\in\Gamma^\infty(T_M),
\end{equation}
that is complex linear in both variables, $C^\infty$-linear in the first variable,
and satisfies
\[
 \nabla_U(aV)=U(a)V+a\nabla_UV
\]
for every smooth complex-valued function~$a$.

Let $f\colon U\to M$, $U\subset\C$, be a holomorphic curve.  Define the
covariant derivatives along~$f$ recursively by
\[
 F_1\coloneqq f',\qquad F_{k+1}\coloneqq\nabla_{f'}F_k\;\;(k\ge1),
\]
and set
\[
 W(\nabla,f)\coloneqq F_1\wedge\cdots\wedge F_n\;\in\; f^*K_M^* .
\]
 Here
\[
 K_M\coloneqq\bigwedge^{\!n}T_M^* \qquad\text{and}\qquad
 K_M^*=\bigwedge^{\!n}T_M
\]
are the canonical and anticanonical bundles; thus the wedge product of the
$n$ covariant derivatives naturally takes values in $f^*K_M^*$.

The expression $\log|W(\nabla,f)|$ is understood locally and does not require a
Hermitian metric.  On a sufficiently small disk whose image lies in a holomorphic
trivializing neighbourhood of $K_M^*$, choose a holomorphic frame~$\varepsilon$
and write $W(\nabla,f)=w_f\varepsilon$.  If $\widetilde\varepsilon=g\varepsilon$ is
another holomorphic frame with $g$ nowhere vanishing, then
\begin{equation}\label{change of frame}
 \widetilde w_f=g(f)^{-1}w_f,\qquad
 \log|\widetilde w_f|=\log|w_f|-\log|g\circ f|.
\end{equation}
The last term is harmonic; consequently holomorphicity of the scalar coefficient
and subharmonicity of its logarithmic modulus are independent of the chosen
holomorphic frame.

For a divisor~$D$
we write $L(D)\coloneqq\mathcal O_M(D)$.  A divisor has \emph{simple normal crossings} if
locally it can be defined by an equation $z^1\cdots z^k=0$ in suitable holomorphic
coordinates $(z^1,\dots,z^n)$.  A submanifold $N\subset M$ is \emph{$\nabla$-totally
geodesic} if, for any local extensions $X,Y$ of tangent vector fields to~$M$,
the restriction of $\nabla_XY$ to~$N$ is again tangent to~$N$ (the normal component
is independent of the chosen extensions).

The Nevanlinna characteristic $T_f(r,L)$ measures the growth of the curve with
respect to the line bundle~$L$, the truncated counting function
$N_n(r,f^*D_i)$ records the intersections with~$D_i$ (multiplicities are cut off
at~$n$), and $S_f(r)$ denotes the standard small error term; for a comprehensive
treatment we refer to~\cite{NW2014,Ru2021}.

We recall the principal theorem of Noguchi~\cite{Noguchi2011}.

\begin{theorem}\label{thm:noguchi}
Let $f\colon\C\to M$ be $\nabla$-nondegenerate, i.e.\ $W(\nabla,f)\not\equiv0$,
and $D=\sum_i D_i$ be an effective reduced divisor with only simple normal
crossings.  Assume that
\begin{enumerate}[label=\textup{(\roman*)}]
\item $\log|W(\nabla,f)|$ is subharmonic;
\item every irreducible component $D_i$ is $\nabla$-totally geodesic.
\end{enumerate}
Then
\[
 T_f(r,L(D))+T_f(r,K_M)
 \le \sum_i N_n(r,f^*D_i)+S_f(r) \qquad \parallel,
\]
where ``$\parallel$'' means that the estimate holds for all $r>0$ outside a set of
finite Lebesgue measure.
\end{theorem}

The idea of using connection-theoretic Wronskians to prove second main theorems
had already appeared in earlier works; we refer, for instance,
to~\cite{Siu1990, Tiba2012}.  It can be traced back to the Weyl--Ahlfors theory of
derived curves~\cite{Ahlfors1941, HuynhXie2022}.  Noguchi's insight was to
formulate a clean curvature positivity condition that makes the argument
completely general and coordinate free.

For projective space equipped with the Fubini--Study connection, Noguchi proves
that the Wronskian is holomorphic and that the totally geodesic submanifolds are
the projective linear subspaces.  The resulting special case recovers Cartan's
Second Main Theorem, giving a geometric interpretation of Cartan's result through
connections and Wronskians~\cite{Noguchi2011}.

At the end of his paper, Noguchi poses the following question:
\begin{question}[{\cite[p.~179]{Noguchi2011}}]\rm \label{Noguchi's question}
Find more examples to which Theorem~\ref{thm:noguchi} applies.
\end{question}

The scope of Theorem~\ref{thm:noguchi} as a general method depends strongly on the availability of such examples.  The Green--Griffiths conjecture~\cite{GreenGriffiths1980} predicts that every entire curve $f\colon\C\to X$ in a projective manifold
$X$ of general type must be algebraically degenerate, i.e.\ its image is contained in a proper algebraic subvariety of $X$.  If Wronskian curvature positivity (already with $D=\emptyset$) were available on a broad class of
varieties of general type, Noguchi's second main theorem would provide a powerful approach to this conjecture.  Our main result shows, however, that the Wronskian condition is highly restrictive: it is equivalent, at the level of existence, to admitting a holomorphic projective connection.  Thus the range of varieties to which this strategy can apply is much narrower than one might initially expect.

\subsection{Basic definitions}
We always assume $\dim_{\C}X = n \ge 2$.  

\begin{definition}[Wronskian curvature positivity]\label{def:wcp}
Let $X$ be a complex manifold of dimension $n\ge2$ and $\nabla$ a
$C^\infty$ connection in $T_X\coloneqq T^{1,0}X$ as introduced
in~\eqref{Noguchi connection}.
For every holomorphic disk $f\colon\Delta\to X$, every sufficiently small
subdisk $\Delta_0\subset\Delta$ whose image lies in a holomorphic trivializing
neighbourhood of $K_X^*$, and every holomorphic frame $\varepsilon$ on that
neighbourhood, write
\[
 W(\nabla,f)|_{\Delta_0}=w_f\,\varepsilon .
\]
We say that $\nabla$ satisfies \emph{Wronskian curvature positivity} if,
whenever $W(\nabla,f)\not\equiv0$, the function $\log|w_f|$ (with the value
$-\infty$ allowed at zeros of $w_f$) is subharmonic on $\Delta_0$.
\end{definition}

The change-of-frame formula~\eqref{change of frame} shows that this condition is
independent of the chosen holomorphic frame; in particular it does not involve a
Hermitian metric.

\medskip

The \emph{torsion} $T^\nabla$ of a connection is given by
\[
T^\nabla(U,V)\coloneqq\nabla_U V-\nabla_V U-[U,V].
\]
$\nabla$ is \emph{torsion-free} when $T^\nabla=0$.  For an arbitrary connection,
its \emph{symmetrization}
\[
\nabla^S_U V \coloneqq \nabla_U V-\frac12\,T^\nabla(U,V)
           = \frac12\bigl(\nabla_U V+\nabla_V U+[U,V]\bigr)
\]
is a globally well-defined torsion-free connection.  In local holomorphic coordinates,
writing $\nabla_{\partial_i}\partial_j=\Gamma^k_{ij}\partial_k$, the coefficients of
$\nabla^S$ are
\[
S^k_{ij}\coloneqq \frac12\bigl(\Gamma^k_{ij}+\Gamma^k_{ji}\bigr).
\]

Two torsion-free connections $D$ and $\widehat D$ are \emph{projectively equivalent}
if locally there exists a smooth $1$-form $\alpha\in T_X^*$ such that
\begin{equation}\label{eq:projective-change}
 \widehat D_U V = D_U V + \alpha(U)V + \alpha(V)U .
\end{equation}
For a torsion-free connection with local coefficients $S^k_{ij}$, its
\emph{projective Christoffel symbols} are
\begin{equation}\label{eq:Pi}
 \Pi^k_{ij}
 \coloneqq S^k_{ij}
 - \frac1{n+1}\bigl(
   \delta_i^k S^a_{aj} + \delta_j^k S^a_{ai}
   \bigr),
\end{equation}
where we use the Einstein summation convention.  The extra term
$\alpha(U)V+\alpha(V)U$ in \eqref{eq:projective-change} is a \emph{pure trace} term:
as a $(1,2)$-tensor its trace with respect to the first two arguments equals
$(n+1)\alpha$, and the whole tensor is completely determined by this trace.
Such a term disappears from the projective Christoffel symbols \eqref{eq:Pi}.
The projective class is called \emph{holomorphic} if
\[
\dbar\Pi^k_{ij}=0
\]
in any holomorphic coordinate chart.  This is a coordinate-free condition:
under a holomorphic change of coordinates the inhomogeneous part of the
transformation law of $\Pi$ is holomorphic, so it vanishes after applying $\dbar$.

In this paper, a \emph{holomorphic projective connection} means a holomorphic
projective class in the above sense.  Equivalently, it is given locally by
torsion-free holomorphic connections whose differences on overlaps have the
pure-trace form \eqref{eq:projective-change}.  This is the standard local
description of a holomorphic normal projective connection.  For compact
K\"ahler manifolds, the equivalent Atiyah-class formulation appears in
\cite[Definition~1.12]{JahnkeRadloff2015} and the discussion that follows it.

\subsection{Main results}
\begin{theorem}\label{thm:existence-equivalence}
Let $X$ be a complex manifold of dimension $n\geq2$. The following conditions are equivalent.
\begin{enumerate}[label=\textup{(\alph*)}]
\item $X$ admits a $C^\infty$ connection satisfying Wronskian curvature positivity.
\item $X$ admits a torsion-free $C^\infty$ connection satisfying Wronskian curvature positivity.
\item $X$ admits a holomorphic projective connection.
\end{enumerate}
\end{theorem}

This result fully answers Noguchi's Question~\ref{Noguchi's question}: a manifold
satisfies his condition precisely when it carries a holomorphic projective
connection.

For a fixed torsion-free connection, the equivalence takes the following form.

\begin{theorem}\label{thm:connection-level}
Let $\nabla$ be a torsion-free $C^\infty$ connection in $T_X$. Then $\nabla$
satisfies Wronskian curvature positivity if and only if its projective class is
holomorphic.
\end{theorem}

If the connection has torsion, the ``only if'' part still holds after passing to
the symmetrization.

\begin{theorem}\label{thm:torsion-forward}
Let $\nabla$ be an arbitrary $C^\infty$ connection in $T_X$, and $\nabla^S$ be
its symmetrization. If $\nabla$ satisfies Wronskian curvature positivity, then the
projective class of $\nabla^S$ is holomorphic.
\end{theorem}

Together with~\cite[Corollaries~3.5--3.6]{JahnkeRadloff2015}, this gives the
classification in the general type case.

\begin{theorem}\label{thm:general-type}
Let $X$ be a smooth projective manifold of general type.  Then the following
conditions are equivalent:
\begin{enumerate}
\item $X$ admits a smooth connection with Wronskian curvature positivity.
\item $X$ admits a holomorphic projective connection.
\item $X$ is a compact ball quotient.
\end{enumerate}
\end{theorem}

In the compact K\"ahler--Einstein case, Theorem~\ref{thm:existence-equivalence}
and a theorem of Jahnke and Radloff~\cite[Theorem~0.1]{JahnkeRadloff2015}
give the three possibilities
\[
\mathbb P^n,\qquad
\text{a finite \'{e}tale quotient of a complex torus},\qquad
\text{or a compact ball quotient}.
\]
Outside the K\"ahler--Einstein setting, Jahnke and Radloff construct flat
holomorphic projective connections on certain modular abelian families over
quaternionic Shimura curves~\cite[Theorem~0.2]{JahnkeRadloff2015}.
For smooth projective manifolds of general type,
Theorem~\ref{thm:general-type} confines an approach to the Green--Griffiths
conjecture through Wronskian curvature positivity to compact ball quotients.

The Fermat quintic surface $X\subset\mathbb P^3$ was one of the motivating
examples.  Using the methods developed in~\cite{HHMX2025, HHMX2026}, one can
show that $H^0(X,\,E_{2,3}T^*X)=0$, where $E_{2,3}T^*X$ denotes the bundle of
invariant $2$-jet differentials of weighted degree~$3$.  If $X$ admitted a
connection with Wronskian curvature positivity, Noguchi's
Theorem~\ref{thm:noguchi} with $D=\emptyset$ and the ampleness of
$K_X\simeq\mathcal O_X(1)$ would imply $W(\nabla,f)\equiv0$ for every entire
curve $f\colon\mathbb C\to X$.  On a surface this says only
$f'\wedge\nabla_{f'}f'=0$: the acceleration is proportional to the velocity
away from critical points.

To prove Theorem~\ref{thm:torsion-forward}, we fix a point and choose
holomorphic coordinates in which the symmetric Christoffel symbols vanish
there.  The covariant jets of a holomorphic curve depend triangularly on
its ordinary coordinate jets, with the highest ordinary jet occurring
with coefficient one.  Keeping the lower covariant jets fixed, we vary
the next ordinary jet by an arbitrary vector $q$.  Subharmonicity for
every magnitude and phase of $q$ forces the coefficient of the resulting
real-linear highest-jet term to vanish.  In dimension two this gives
\[
(\dbar S)_{\bar v}(v,v)\in\C v,
\]
while in higher dimension the same conclusion follows from the variation
of an $(n-1)$-fold wedge.  The pure-trace lemma then gives $\dbar\Pi=0$.

Conversely, projectively equivalent torsion-free connections produce
covariant jets related by a triangular transformation with diagonal
entries equal to one, so their Wronskians coincide exactly.  A local
holomorphic representative therefore gives a holomorphic Wronskian.
A partition of unity combines these representatives into a global
smooth connection in the same projective class, satisfying Wronskian
curvature positivity.

Sections~\ref{sec:surface-proof} and~\ref{sec:higher-proof} prove the forward
implication, and Section~\ref{sec:converse} proves the converse.
Section~\ref{sect: 6} establishes Theorem~\ref{thm:general-type}, the
K\"ahler--Einstein trichotomy (Theorem~\ref{cor:KE}), and a Chern-class
obstruction for smooth nonlinear projective hypersurfaces.

\section{Preliminaries}\label{sect: 2}

\subsection{The local Laplacian formula}\label{sec:localmeaning}
For a local holomorphic curve $f:\Delta\to X$, set
\begin{equation}
    \label{F_i definition}
 F_1\coloneqq f',\qquad F_{r+1}\coloneqq\nabla_{f'}F_r\quad(r\geq1).
\end{equation}
The Wronskian is
\[
 W(\nabla,f)=F_1\wedge\cdots\wedge F_n\in f^*K_X^*.
\]
On a sufficiently small subdisk whose image lies in a holomorphic trivializing
neighbourhood of $K_X^*$, write $W(\nabla,f)=w_f\varepsilon$ in a holomorphic
frame $\varepsilon$.  The coefficient $w_f$ is smooth, and wherever $w_f\neq0$,
\begin{equation}\label{eq:log-derivative}
 \partial_t\partial_{\bar t}\log|w_f|^2
 =2\operatorname{Re}\left(
 \frac{(w_f)_{t\bar t}}{w_f}
 -\frac{(w_f)_t(w_f)_{\bar t}}{w_f^2}
 \right).
\end{equation}
Since $\Delta=4\partial_t\partial_{\bar t}$ and
$\log|w_f|^2=2\log|w_f|$, subharmonicity of $\log|w_f|$ implies that
\eqref{eq:log-derivative} is nonnegative.  The forward implication uses only
this inequality at points where $w_f\neq0$.

\subsection{Local holomorphic representatives}
\begin{lemma}\label{lem:local-holomorphic-representative}
A projective class of torsion-free smooth connections is holomorphic if and only if it admits, locally near every point, a torsion-free holomorphic representative.
\end{lemma}

\begin{proof}
A holomorphic representative has holomorphic projective Christoffel symbols.
Conversely, if $\dbar\Pi=0$ in a holomorphic chart, let $D$ have coefficients
$D^k_{ij}\coloneqq\Pi^k_{ij}$.  These are symmetric and holomorphic, so $D$ is a
torsion-free holomorphic connection.  Since $\Pi^a_{aj}=0$, its projective
Christoffel symbols are $\Pi$.  Equivalently, for the original connection $S$,
\[
 S^k_{ij}-D^k_{ij}
 =\frac1{n+1}\bigl(\delta_i^k S^a_{aj}+\delta_j^k S^a_{ai}\bigr),
\]
so $D$ belongs to the given projective class.
\end{proof}

\subsection{Normal coordinates for the symmetric part}
\begin{lemma}\label{lem:normal}
Let $p\in X$. There exist local holomorphic coordinates $(u^1,\dots,u^n)$ centered at $p$ such that
\[
 S^k_{ij}(p)=0
\]
for all $i,j,k$. In particular, if the connection is torsion-free, this means $\Gamma^k_{ij}(p)=0$ for all $i,j,k$.
\end{lemma}

\begin{proof}
Starting with holomorphic coordinates $z$ centered at $p$, consider the change of variables
\[
 z^k=u^k-\frac12 S^k_{ij}(p)u^i u^j.
\]
The Jacobian at $0$ is the identity, so the holomorphic inverse function theorem
gives local coordinates $u$.  The transformation law for the symmetric
connection coefficients, evaluated at $p$, gives
\[
 \widetilde S^k_{ij}(p)
 =S^k_{ij}(p)+\frac{\partial^2z^k}{\partial u^i\partial u^j}(0)=0.
\]
\end{proof}

\begin{remark}\label{rem:no-second-change}
The normalization $S(p)=0$ is pointwise: the derivatives of $S$ at $p$ may be
nonzero, and no local flatness is implied.  We keep these coordinates fixed
throughout the proof and prescribe the jets of test curves without further
coordinate changes.
\end{remark}

\subsection{Prescribing covariant jets}
\begin{lemma}\label{lem:triangular-jets}
Fix a point $p\in X$ and local holomorphic coordinates centered at $p$.
For a holomorphic curve germ $f:(\Delta,0)\to(X,p)$, write
\[
 j_r(t)\coloneqq\left(\frac{d^rf^1}{dt^r}(t),\dots,
                   \frac{d^rf^n}{dt^r}(t)\right),
 \qquad J_r(f)\coloneqq j_r(0)\in T_pX\quad(r\ge1).
\]
Set $F_1\coloneqq f'$ and $F_{r+1}\coloneqq\nabla_{f'}F_r$.  For every $r\ge1$, there are
smooth expressions $P_r$ and $Q_r$, determined by the connection and its
finite jets in the chosen coordinates, such that
\begin{align}
 F_r(0) &= J_r(f)+P_r\bigl(J_1(f),\dots,J_{r-1}(f)\bigr),
 \label{eq:triangular-Fr}\\
 \partial_tF_r(0) &= J_{r+1}(f)+Q_r\bigl(J_1(f),\dots,J_r(f)\bigr).
 \label{eq:triangular-dFr}
\end{align}
Thus the highest ordinary jet in each formula occurs with coefficient one.
In particular, any prescribed vectors $E_1,\dots,E_N\in T_pX$ can be
realized as $F_r(0)=E_r$ for $1\le r\le N$ by a holomorphic polynomial
curve germ.

Fix $m\ge1$.  If two holomorphic curve germs through $p$ satisfy
\[
 J_r(\widetilde f)=J_r(f)\quad(1\le r\le m),\qquad
 J_{m+1}(\widetilde f)=J_{m+1}(f)+q,
 \quad q\in T_pX,
\]
then
\begin{align*}
 \widetilde F_r(0)&=F_r(0) &&(1\le r\le m),\\
 \partial_t\widetilde F_r(0)&=\partial_tF_r(0) &&(1\le r<m),\\
 \partial_t\widetilde F_m(0)&=\partial_tF_m(0)+q.
\end{align*}
The antiholomorphic and mixed derivatives also remain unchanged:
\[
 \partial_{\bar t}\widetilde F_r(0)=\partial_{\bar t}F_r(0)
 \quad(1\le r\le m),\qquad
 \partial_t\partial_{\bar t}\widetilde F_m(0)
 =\partial_t\partial_{\bar t}F_m(0).
\]
\end{lemma}

\begin{proof}
The identities $F_1=j_1$ and
\[
 F_{r+1}=\partial_tF_r+\Gamma(f)(j_1,F_r)
\]
show inductively that $F_r$ is $j_r$ plus an expression involving only
$j_1,\dots,j_{r-1}$ and derivatives of the connection coefficients at
$f(t)$.  Indeed, $\partial_tj_r=j_{r+1}$, while all remaining terms in
$F_{r+1}$ use jets of order at most $r$.  Evaluating at $0$ gives
\eqref{eq:triangular-Fr}; differentiating once gives
\eqref{eq:triangular-dFr}.

To prescribe the covariant jets, choose the ordinary jets recursively by
\[
 A_1\coloneqq E_1,\qquad A_r\coloneqq E_r-P_r(A_1,\dots,A_{r-1})\quad(2\le r\le N).
\]
A holomorphic polynomial curve with $J_r(f)=A_r$ then has $F_r(0)=E_r$.
The assertions about $F_r(0)$ and $\partial_tF_r(0)$ under variation of
$J_{m+1}$ follow directly from \eqref{eq:triangular-Fr} and
\eqref{eq:triangular-dFr}.

Finally, all $j_s$ are holomorphic.  In the lower-order expression for
$F_r$, the derivative $\partial_{\bar t}$ therefore acts only on the
connection coefficients and their derivatives, introducing
$\overline{j_1}$.  Hence $\partial_{\bar t}F_r(0)$ for $r\le m$ depends
only on $J_1,\dots,J_r$.  A further $\partial_t$-derivative raises the
order of an ordinary jet by at most one and annihilates
$\overline{j_1}$.  Since the lower-order part of $F_m$ uses only
$j_1,\dots,j_{m-1}$, the mixed derivative
$\partial_t\partial_{\bar t}F_m(0)$ depends only on $J_1,\dots,J_m$.
Neither derivative involves the free jet $J_{m+1}$.
\end{proof}

\subsection{Three linear-algebra lemmas}

\begin{lemma}\label{lem:scaling}
Let $V$ be a complex vector space, $L:V\to\C$ a complex linear functional, and $c\in\R$. If
\[
c+2\operatorname{Re}L(q)\ge 0
\]
for all $q\in V$, then $L=0$.
\end{lemma}

\begin{proof}
If $L(q_0)\neq0$, complex multiples of $q_0$ make
$\operatorname{Re}L(q)$ arbitrarily negative, contradicting the inequality.
\end{proof}

\begin{lemma}\label{lem:wedge-nondegenerate}
Let $V$ be a finite-dimensional complex vector space of dimension $m\ge n$ and $B\in\bigwedge^{n-1}V$.
If $B\wedge q = 0$ for every $q\in V$, then $B=0$.
\end{lemma}

\begin{proof}
In a basis $e_1,\dots,e_m$, write $B=\sum_{|I|=n-1}B^I e_I$,
where $e_I\coloneqq\bigwedge_{i\in I}e_i$.
If $B^I\neq0$, let $J$ be the nonempty complement of $I$.
The hypothesis gives $B\wedge e_J=0$.  But every term of $B\wedge e_J$
except the one indexed by $I$ has a repeated basis vector, so
\[
 B\wedge e_J=\pm B^I e_1\wedge\cdots\wedge e_m\neq0,
\]
a contradiction.
\end{proof}

\begin{lemma}[Pure trace]\label{lem:pure-trace}
Let $V$ be a complex vector space of dimension $n\ge 2$ and
$B\in\Sym^2V^*\otimes V$.
If
\[
 B(v,v)\in\C v \qquad\text{for every } v\in V ,
\]
then there exists a unique linear form $\beta\in V^*$ such that
\begin{equation}\label{eq:pure-trace}
 B(u,v)=\beta(u)v+\beta(v)u \qquad\text{for all } u,v\in V .
\end{equation}
\end{lemma}

\begin{proof}
Fix a basis $e_1,\dots,e_n$, and write $B(e_i,e_i)=\lambda_i e_i$.
For $i\neq j$ and $t\in\C$, the hypothesis gives
\begin{equation}\label{eq:pure-trace-expand}
\lambda_i e_i+2t B(e_i,e_j)+t^2\lambda_j e_j
=\mu(t)(e_i+te_j).
\end{equation}
Taking $t=1$ shows that $B(e_i,e_j)=a_{ij}e_i+a_{ji}e_j$.
Comparing the two coordinates in \eqref{eq:pure-trace-expand} and eliminating
$\mu(t)$ gives
\[
 2t a_{ji}+t^2\lambda_j=t(\lambda_i+2t a_{ij}).
\]
Thus $2a_{ji}=\lambda_i$ and $2a_{ij}=\lambda_j$.  With
$\beta(e_i)\coloneqq\lambda_i/2$, bilinearity yields \eqref{eq:pure-trace}.
Uniqueness follows by comparing the coefficients of any linearly independent
pair $u,v$ in that identity.
\end{proof}

\section{Positivity forces a holomorphic projective class in dimension two}\label{sec:surface-proof}

Let $X$ be a complex surface and $\nabla$ a smooth connection satisfying
Wronskian curvature positivity.  Since $T^\nabla(f',f')=0$, the two factors
$F_1=f'$ and $F_2=\nabla_{f'}f'$ in the Wronskian are unchanged by symmetrization.
We therefore work with the symmetric coefficients $S^k_{ij}$.

Fix $p\in X$ and choose normal coordinates as in Lemma~\ref{lem:normal}.  Set
\[
 C^k_{\bar m ij} \coloneqq \frac{\partial S^k_{ij}}{\partial\bar z^m}(p),
\qquad
 C_{\bar v}(v,v)^k \coloneqq C^k_{\bar m ij}\,\overline{v^m}v^iv^j \quad (v\in T_pX).
\]

Pick a holomorphic curve germ through $p$ such that
\[
 F_1(0)=v,\qquad F_2(0)=a,
\]
with $v$ and $a$ linearly independent.  In the local frame $\partial_1\wedge\partial_2$
of $K_X^*$ we write
\[
 W(\nabla,f)=w(t)\,\partial_1\wedge\partial_2 .
\]
Clearly $w(0)=\det(v,a)\neq0$.  Because $F_1=f'$ is holomorphic, $\partial_{\bar t}F_1(0)=0$.

To compute $\partial_{\bar t}F_2(0)$, recall that in coordinates
\[
 F_2^k = f''{}^k + S^k_{ij}(f)\,f'{}^i f'{}^j .
\]
The terms $f''{}^k$ and $f'{}^i f'{}^j$ are holomorphic, so their $\bar t$-derivatives vanish.
Moreover, $S(p)=0$ by our choice of normal coordinates; differentiating the second term
thus gives
\[
\partial_{\bar t}\bigl(S^k_{ij}(f)f'{}^i f'{}^j\bigr)\big|_{t=0}
= \frac{\partial S^k_{ij}}{\partial\bar z^m}(p)\,
  \overline{f'{}^m(0)}\, f'{}^i(0) f'{}^j(0)
= C^k_{\bar m ij}\,\overline{v^m}v^iv^j .
\]
We introduce the tensor $C_{\bar m}\in \Sym^2 T_p^*X \otimes T_pX$ whose components
are $C_{\bar m\, ij}^k$, and write $C_{\bar m}(v,v)$ for the vector with these components.
With this notation,
\[
 \partial_{\bar t}F_2(0) = \sum_m \overline{v^m}\, C_{\bar m}(v,v)
 = C_{\bar v}(v,v).
\]

Now $w(t) = \det(F_1(t),F_2(t))$, and therefore
\begin{equation}\label{eq:surface-wbar}
 w_{\bar t}(0) = \det\bigl(v,\,C_{\bar v}(v,v)\bigr).
\end{equation}

We now exploit the freedom of the third-order jet.  Vary only the ordinary jet
$J_3 = f'''(0)$ by adding a fixed vector $q\in T_pX$.  By Lemma~\ref{lem:triangular-jets},
this variation leaves $v$ and $a$ unchanged and adds exactly $q$ to $\partial_tF_2(0)$.
Differentiating $w(t)=\det(F_1,F_2)$ gives
\[
 w_t(0) = \det(\partial_tF_1(0),F_2(0)) + \det(F_1(0),\partial_tF_2(0)).
\]
Because $S(p)=0$, we have $F_2(0)=f''(0)=\partial_tF_1(0)=a$, so the first determinant
$\det(a,a)$ vanishes.  Thus $w_t(0) = \det(v,\partial_tF_2(0))$.  The $q$-linear part of
$\partial_tF_2(0)$ is $q$, so the $q$-linear part of $w_t(0)$ is $\det(v,q)$.

Lemma~\ref{lem:triangular-jets} also shows that $\partial_{\bar t}F_2(0)$ and
$\partial_t\partial_{\bar t}F_2(0)$ are independent of $q$.  Since
$\partial_{\bar t}F_1=0$, both $w_{\bar t}(0)$ and $w_{t\bar t}(0)$ are
independent of $q$.  Thus the only $q$-dependent term in
\eqref{eq:log-derivative} at $0$ is
\[
 -2\operatorname{Re}\!\left(
 \frac{\det(v,q)\,w_{\bar t}(0)}{w(0)^2}
 \right).
\]

Each $q$ is realized by a polynomial curve germ on a disk whose radius may
depend on $q$.  Since $w(0)=\det(v,a)\ne0$ is fixed, Wronskian curvature
positivity gives
\[
 c-2\operatorname{Re}\!\left(
 \frac{\det(v,q)\,w_{\bar t}(0)}{w(0)^2}
 \right) \ge 0
 \qquad\text{for all } q\in T_pX ,
\]
where $c\in\mathbb R$ is independent of $q$.
Apply Lemma~\ref{lem:scaling} to the complex linear functional
\[
 L(q)\coloneqq-\frac{\det(v,q)\,w_{\bar t}(0)}{w(0)^2}.
\]
It follows that $L$ vanishes identically.  Because $v\neq0$, the functional $q\mapsto\det(v,q)$ is
nonzero, and we conclude $w_{\bar t}(0)=0$.  Substituting this into
\eqref{eq:surface-wbar} yields
\[
 \det\!\bigl(v,\,C_{\bar v}(v,v)\bigr)=0 \qquad\text{for all } v\in T_pX .
\]
In a two-dimensional space, this means $C_{\bar v}(v,v)$ is proportional to $v$:
\[
 C_{\bar v}(v,v) \in \C v \qquad (v\in T_pX). \tag{*}
\]

The conjugate parameter can now be removed.  Condition~($*$) is equivalent to
$v\wedge C_{\bar v}(v,v)=0$, i.e.
\[
\sum_m \overline{v^m}\, \bigl(v\wedge C_{\bar m}(v,v)\bigr) = 0 .
\]
The left-hand side is a polynomial in the components of $v$ and their complex conjugates.
After complexification, $v^m$ and $\overline{v^m}$ become independent variables,
so each term $v\wedge C_{\bar m}(v,v)$ must vanish separately.  Hence,
\[
 v\wedge C_{\bar m}(v,v) = 0 \qquad\text{for each fixed $m$ and all } v\in T_pX .
\]
Thus $C_{\bar m}(v,v)$ is also parallel to $v$.  Because $S^k_{ij}$ is symmetric
in $i,j$, each $C_{\bar m}$ belongs to $\Sym^2(T_p^*X)\otimes T_pX$.
Lemma~\ref{lem:pure-trace} therefore provides linear forms $\beta_{\bar m}\in T_p^*X$
such that
\begin{equation}\label{eq:C-pure-trace}
 C^k_{\bar m ij} = \delta_i^k\beta_{j\bar m} + \delta_j^k\beta_{i\bar m}.
\end{equation}

Finally, differentiate the projective Christoffel symbols \eqref{eq:Pi} with respect
to $\bar z^m$.  Using \eqref{eq:C-pure-trace} and the trace relation
$C^a_{\bar m aj} = (n+1)\beta_{j\bar m}$ (which in dimension two reads
$C^a_{\bar m aj}=3\beta_{j\bar m}$), we obtain
\[
 \frac{\partial\Pi^k_{ij}}{\partial\bar z^m}
 = C^k_{\bar m ij}
   - \frac1{n+1}\bigl(\delta_i^k C^a_{\bar m aj} + \delta_j^k C^a_{\bar m ai}\bigr)
 = 0 .
\]
Thus $\dbar\Pi(p)=0$ at the chosen point.  Under a holomorphic change of coordinates,
the inhomogeneous part of the transformation law of $\Pi$ is holomorphic; consequently
$\dbar\Pi$ transforms tensorially.  Since $p$ was arbitrary, $\dbar\Pi=0$ on $X$.
Hence the projective class of $\nabla^S$ is holomorphic.

\section{Positivity forces a holomorphic projective class in higher dimension}\label{sec:higher-proof}

Now assume $n\geq3$.  Let $\nabla$ be an arbitrary smooth connection satisfying
Wronskian curvature positivity.  Fix $p\in X$ and, using Lemma~\ref{lem:normal},
choose normal coordinates for the symmetric part so that $S(p)=0$.

\medskip
\noindent\textbf{Step 1: Setting up a convenient curve.}
By Lemma~\ref{lem:triangular-jets} we can pick a holomorphic curve germ through~$p$
whose covariant jets
\[
E_r\coloneqq F_r(0)\qquad (1\le r\le n)
\]
form a basis of $V\coloneqq T_pX$.  Along the curve define
\[
\mathcal U(t)\coloneqq F_1(t)\wedge\cdots\wedge F_{n-1}(t),\qquad
U_0\coloneqq\mathcal U(0)=E_1\wedge\cdots\wedge E_{n-1},
\]
and set $H\coloneqq\Span(E_1,\dots,E_{n-1})\subset V$.  In a local holomorphic frame
of $K_X^*$ we write the Wronskian as $W(\nabla,f)=w(t)\,\text{(frame)}$, so that
\[
w(0)=E_1\wedge\cdots\wedge E_n\neq0,
\]
identifying an $n$-vector with its scalar coefficient.

\medskip
\noindent\textbf{Step 2: Variation of the highest jet.}
We now vary only the ordinary $(n+1)$-st jet $J_{n+1}=f^{(n+1)}(0)$ by adding an
arbitrary vector $q\in V$.  Lemma~\ref{lem:triangular-jets} tells us that this
variation leaves all $F_r(0)$ ($1\le r\le n$) unchanged, while $\partial_tF_n(0)$
acquires exactly the term $q$.  No other $\partial_tF_j(0)$ depends on $q$.
Consequently, the $q$-linear part of $w_t(0)$ is
\begin{equation}\label{eq:qt-wt}
 U_0\wedge q .
\end{equation}
Set $D_j\coloneqq\partial_{\bar t}F_j(0)$.  By Lemma~\ref{lem:triangular-jets}, the
$D_j$ and $\partial_t\partial_{\bar t}F_n(0)$ are independent of $q$.
Differentiating $w=\mathcal U\wedge F_n$ therefore shows that the
$q$-linear part of $w_{t\bar t}(0)$ is
\begin{equation}\label{eq:qt-wtb}
 \partial_{\bar t}\mathcal U(0)\wedge q .
\end{equation}
The quantities $w(0)$ and $w_{\bar t}(0)$ are also independent of $q$.

\medskip
\noindent\textbf{Step 3: Using subharmonicity.}
Substituting \eqref{eq:qt-wt} and \eqref{eq:qt-wtb} into
\eqref{eq:log-derivative}, we find that the $q$-linear contribution is
\[
2\operatorname{Re}\!\left[
  \frac{\bigl((\partial_{\bar t}\mathcal U)(0)-\lambda U_0\bigr)\wedge q}{w(0)}
\right],
\qquad
\lambda\coloneqq\frac{w_{\bar t}(0)}{w(0)}.
\]
The remaining terms are independent of $q$.

As in the surface case, each $q$ is realized on a sufficiently small disk,
and $w(0)\ne0$ is fixed.  Wronskian curvature positivity at $0$ gives
the inequality below for a real constant $c$ independent of $q$:
\[
c+2\operatorname{Re}\!\left[
\frac{\bigl((\partial_{\bar t}\mathcal U)(0)-\lambda U_0\bigr)\wedge q}{w(0)}
\right]\ge0
\qquad\text{for all }q\in V.
\]
Lemma~\ref{lem:scaling} therefore shows that the complex linear functional
inside the real part vanishes identically.  Lemma~\ref{lem:wedge-nondegenerate} then implies
\begin{equation}\label{eq:Uparallel}
 (\partial_{\bar t}\mathcal U)(0)=\lambda U_0 .
\end{equation}

\medskip
\noindent\textbf{Step 4: Geometric interpretation.}
Expand $\partial_{\bar t}\mathcal U(0)$ using the Leibniz rule:
\[
(\partial_{\bar t}\mathcal U)(0)
 = \sum_{j=1}^{n-1} E_1\wedge\cdots\wedge D_j\wedge\cdots\wedge E_{n-1}.
\]
Pick any vector $N\notin H$ and decompose $D_j = D_j^H + c_jN$ with $D_j^H\in H$.
The components along $N$ contribute
\[
c_j\;E_1\wedge\cdots\wedge N\wedge\cdots\wedge E_{n-1},
\]
where $N$ occupies the $j$-th slot.  Together with $U_0$, these $(n-1)$-vectors
form a basis of $\bigwedge^{n-1}V$.  Since the right-hand side of \eqref{eq:Uparallel}
lies in $\mathbb{C}\,U_0$, all coefficients $c_j$ must vanish.  Consequently,
\begin{equation}\label{eq:Dj-in-H}
 D_j\in H\qquad (1\le j\le n-1).
\end{equation}

\medskip
\noindent\textbf{Step 5: The condition on $D_2$.}
In particular $D_2\in H$.  Recall that in coordinates
\[
F_2 = f'' + S(f)(f',f'),
\]
since the skew part of $\Gamma$ vanishes on $(f',f')$.
With $v\coloneqq E_1=f'(0)$, holomorphicity of $f'$ and $f''$ gives
\[
D_2 = (\dbar S)_p(\bar v; v,v).
\]
In components, set
$C^k_{\bar m ij}\coloneqq\frac{\partial S^k_{ij}}{\partial\bar z^m}(p)$.
Then
\[
C_{\bar v}(v,v)\coloneqq\sum_m \overline{v^m}\, C_{\bar m}(v,v) = D_2.
\]

\medskip
\noindent\textbf{Step 6: From one hyperplane to all hyperplanes.}
Fix $v\neq0$, and let $H\subset V$ be an arbitrary hyperplane containing
$\mathbb C v$.  Choose $E_2,\dots,E_{n-1}$ so that
$v,E_2,\dots,E_{n-1}$ is a basis of $H$, and then choose $E_n\notin H$.
Lemma~\ref{lem:triangular-jets} realizes these vectors as the covariant jets of a
holomorphic curve germ.  Equation~\eqref{eq:Dj-in-H} gives
$C_{\bar v}(v,v)=D_2\in H$.  The left-hand side depends only on $v$ and the
connection at $p$, not on the auxiliary choice of $H,E_2,\dots,E_n$.  Since $H$
was arbitrary and the intersection of all hyperplanes containing $\mathbb C v$
is $\mathbb C v$, we obtain
\[
C_{\bar v}(v,v)\in\mathbb{C}\,v \qquad\text{for all } v\in V .
\]

\medskip
\noindent\textbf{Step 7: Elimination of the conjugate variable and conclusion.}
Wedge the preceding relation with $v$.  We obtain the polynomial identity
\[
0=v\wedge C_{\bar v}(v,v)
 =\sum_m \overline{v^m}\,\bigl(v\wedge C_{\bar m}(v,v)\bigr).
\]
After writing $v=x+iy$, this is a vector-valued polynomial identity in the
independent real variables $x$ and $y$; equivalently, after complexification one
may regard $v$ and $\bar v$ as independent.  Therefore the coefficient of each
$\overline{v^m}$ vanishes:
\[
v\wedge C_{\bar m}(v,v)=0
\qquad\text{for every $m$ and every $v\in V$}.
\]
Thus, for each fixed $m$,
\[
C_{\bar m}(v,v)\in\mathbb{C}\,v\qquad (v\in V).
\]
Since each $C_{\bar m}$ is a symmetric tensor, Lemma~\ref{lem:pure-trace} supplies
linear forms $\beta_{\bar m}$ satisfying \eqref{eq:C-pure-trace}.
Computing the trace gives $C^a_{\bar m aj}=(n+1)\beta_{j\bar m}$.
Finally, by~\eqref{eq:Pi} and this trace relation, we obtain
\[
\frac{\partial\Pi^k_{ij}}{\partial\bar z^m}
 = C^k_{\bar m ij}
   -\frac1{n+1}\bigl(\delta_i^k C^a_{\bar m aj}+\delta_j^k C^a_{\bar m ai}\bigr)
 = 0 .
\]
Hence $\dbar\Pi=0$ at $p$, and since $p$ was arbitrary, the projective class
of $\nabla^S$ is holomorphic.  Together with Section~\ref{sec:surface-proof},
this completes the proof of Theorem~\ref{thm:torsion-forward} for all
$n\ge 2$.

\section{From holomorphic projective connections back to positivity}\label{sec:converse}

The converse follows from the invariance of the Wronskian under projective
changes of torsion-free connections.

\subsection{Exact projective invariance of the Wronskian}\label{sec:projective-invariance}

\begin{proposition}\label{prop:projective-invariance}
Let $D$ and $\widehat D$ be projectively equivalent torsion-free connections,
related by~\eqref{eq:projective-change}.  Then for every holomorphic curve $f$,
\[
 W(\widehat D,f)=W(D,f).
\]
\end{proposition}

\begin{proof}
Set $F_1\coloneqq f'$ and $F_{r+1}\coloneqq D_{f'}F_r$, and define $\widehat F_r$ analogously
using $\widehat D$.  We prove by induction on $r$ that
\begin{equation}\label{eq:triangular-projective}
 \widehat F_r = F_r + \sum_{j<r} a_{rj} F_j
\end{equation}
for suitable smooth functions $a_{rj}$ along the curve.

For $r=1$ the statement is trivial.  For $r=2$, using
\eqref{eq:projective-change} we obtain
\[
\widehat F_2 = \widehat D_{f'}f' = D_{f'}f' + 2\alpha(f')f' = F_2 + 2\alpha(f')F_1,
\]
which is of the required form with $a_{21}=2\alpha(f')$.

Assume \eqref{eq:triangular-projective} holds for $\widehat F_r$.  Then
\[
\begin{aligned}
\widehat F_{r+1}
&= \widehat D_{f'}\widehat F_r \\
&= D_{f'}\widehat F_r + \alpha(f')\widehat F_r + \alpha(\widehat F_r)f'
 \qquad\text{(by \eqref{eq:projective-change})}.
\end{aligned}
\]
Insert the inductive hypothesis for $\widehat F_r$ and expand $D_{f'}\widehat F_r$:
\[
\begin{aligned}
D_{f'}\widehat F_r
&= D_{f'}\Bigl(F_r + \sum_{j<r} a_{rj} F_j\Bigr) \\
&= F_{r+1} + \sum_{j<r} (\partial_t a_{rj}) F_j + \sum_{j<r} a_{rj} F_{j+1},
\end{aligned}
\]
where $\partial_t$ denotes the derivative along the curve.  The coefficient of
$F_{r+1}$ is exactly $1$.  All other terms in this expression, together with the two
$\alpha$-terms $\alpha(f')\widehat F_r$ and $\alpha(\widehat F_r)f'$, lie in
$\operatorname{Span}(F_1,\dots,F_r)$.  Hence $\widehat F_{r+1}$ is again of the
form $F_{r+1} + \sum_{j\le r} a_{r+1,j}F_j$, completing the induction.

The triangular relation \eqref{eq:triangular-projective} has diagonal
coefficients equal to $1$, so the wedge products coincide.
\end{proof}

\subsection{Holomorphic representatives yield holomorphic Wronskians}\label{sec:holomorphic-representative}

\begin{proposition}\label{prop:holomorphic-rep}
Let $D$ be a torsion-free holomorphic connection on an open set.  Then for any
holomorphic curve $f$, the section $W(D,f)$ is holomorphic.
\end{proposition}

\begin{proof}
$F_1=f'$ is holomorphic.  If $F_r$ is holomorphic, then
$F_{r+1}=D_{f'}F_r$ is holomorphic because $f$ is holomorphic, the coefficients
of $F_r$ are holomorphic, and the Christoffel symbols of $D$ are holomorphic.
By induction, every $F_r$ is holomorphic; hence their wedge product is holomorphic.
\end{proof}

\noindent
If the projective class of a torsion-free smooth connection $\nabla$ is holomorphic,
Lemma~\ref{lem:local-holomorphic-representative} provides, near every point, a local
torsion-free holomorphic representative $D$ in the same projective class.  Since
$D$ and $\nabla$ are projectively equivalent, Proposition~\ref{prop:projective-invariance}
yields $W(\nabla,f)=W(D,f)$, and Proposition~\ref{prop:holomorphic-rep} implies that
$W(D,f)$ is holomorphic.  The logarithmic modulus of a holomorphic function (allowing
$-\infty$ at zeros) is subharmonic, so $\nabla$ satisfies Wronskian curvature
positivity.  Together with the forward direction proved in
Sections~\ref{sec:surface-proof} and~\ref{sec:higher-proof}, this establishes
Theorem~\ref{thm:connection-level}.

\subsection{Globalisation: from local representatives to a global smooth connection}\label{sec:globalization}

We now assume only that $X$ admits a holomorphic projective connection, i.e.\ a
holomorphic projective class.  Choose an open cover $\{U_\alpha\}$ and, on each
$U_\alpha$, a torsion-free holomorphic representative $D_\alpha$ of that class.
On overlaps $U_\alpha\cap U_\beta$, the difference $D_\alpha-D_\beta$ is a pure-trace
term, so there exists a holomorphic $1$-form $\eta_{\alpha\beta}$ on the overlap
such that
\[
(D_\alpha-D_\beta)_U V = \eta_{\alpha\beta}(U)V + \eta_{\alpha\beta}(V)U .
\]

Let $\{\rho_\alpha\}$ be a locally finite smooth partition of unity subordinate to
the cover, and define a global connection by
\[
\nabla \coloneqq \sum_\alpha \rho_\alpha D_\alpha .
\]
Because $\sum_\alpha \rho_\alpha=1$, this is again a connection; it is torsion-free
since each $D_\alpha$ is.  On a fixed open set $U_\beta$,
\[
\nabla - D_\beta
 = \Bigl(\sum_\alpha \rho_\alpha \eta_{\alpha\beta}\Bigr) \odot \operatorname{id},
\]
where $(\eta\odot\operatorname{id})_{U,V}\coloneqq\eta(U)V+\eta(V)U$.  Thus $\nabla$ differs
from $D_\beta$ by a pure-trace term; consequently $\nabla$ and $D_\beta$ are
projectively equivalent.

Applying Proposition~\ref{prop:projective-invariance} to this projective equivalence
gives $W(\nabla,f)=W(D_\beta,f)$ for any holomorphic disk $f$ whose image lies in
$U_\beta$.  The right-hand side is holomorphic by Proposition~\ref{prop:holomorphic-rep},
so $\log|W(\nabla,f)|$ is subharmonic.  Hence $\nabla$ satisfies Wronskian curvature
positivity.  This proves the implication (c)$\Rightarrow$(b) in
Theorem~\ref{thm:existence-equivalence}.

Finally, (b)$\Rightarrow$(a) is obvious, while (a)$\Rightarrow$(c) is exactly
Theorem~\ref{thm:torsion-forward}.  Therefore the three conditions in
Theorem~\ref{thm:existence-equivalence} are equivalent, and the proof of the main
theorems is complete.

\begin{samepage}
\section{Geometric consequences}
\label{sect: 6}

\subsection{The K\"ahler--Einstein trichotomy}
Theorem~\ref{thm:existence-equivalence} and the classification of holomorphic
projective connections on compact K\"ahler--Einstein manifolds give the following
result.

\begin{theorem}
 \label{cor:KE}
Let $X$ be a compact K\"ahler--Einstein manifold of dimension $n\ge2$.  Then $X$
admits a connection with Wronskian curvature positivity if and only if it is
\begin{itemize}
  \item complex projective space $\Pj^n$, or
  \item a finite \'etale quotient of a compact complex torus (with the induced flat
        projective connection), or
  \item a compact ball quotient.
\end{itemize}
\end{theorem}
\end{samepage}

\begin{proof}
By Theorem~\ref{thm:existence-equivalence}, Wronskian curvature positivity is
equivalent to the existence of a holomorphic projective connection.  The
classification of compact K\"ahler--Einstein manifolds admitting such a connection
is exactly the three cases listed; see~\cite{KobayashiOchiai1980} and
\cite[Theorem~0.1]{JahnkeRadloff2015}.
\end{proof}

These three cases correspond to the three possible signs of the first Chern class:
positive for $\Pj^n$, zero for torus quotients, and negative for ball quotients
(the latter being of general type).  Hence Wronskian curvature positivity is
\emph{not} incompatible with general type, but among K\"ahler--Einstein manifolds
of general type it singles out precisely the ball quotients.

\begin{remark}
Theorem~\ref{cor:KE} should not be misread as a classification of \emph{all}
projective manifolds with holomorphic projective connections.  Away from the
K\"ahler--Einstein setting, additional higher-dimensional examples exist,
for instance certain modular abelian families over quaternionic Shimura curves;
see~\cite[Theorem~0.2]{JahnkeRadloff2015}.
\end{remark}

\subsection{Proof of Theorem~\ref{thm:general-type}}
\label{sec:proof-general-type}

\begin{proof}
The equivalence between (1) and (2) is Theorem~\ref{thm:existence-equivalence}.
Assume that $X$ admits a holomorphic projective connection.  Since $X$ is
of general type, its canonical bundle is big, and in particular
$X\not\simeq \mathbb P^n$.  By \cite[Corollary~3.5]{JahnkeRadloff2015},
the canonical bundle of a projective manifold with a holomorphic projective
connection is nef unless the manifold is projective space.  Hence $K_X$ is
big and nef.  It follows from \cite[Corollary~3.6(1)]{JahnkeRadloff2015} that
$X$ is a compact ball quotient.

Conversely, every compact ball quotient carries a flat holomorphic normal
projective connection, hence a holomorphic projective connection
\cite[Theorem~0.1]{JahnkeRadloff2015}.  Applying
Theorem~\ref{thm:existence-equivalence} gives a smooth connection satisfying
Wronskian curvature positivity.
\end{proof}

\subsection{Why nonlinear hypersurfaces are excluded}\label{sec:hypersurfaces}
A holomorphic projective connection imposes a strong Chern-class identity.  Let
$c_i(X)\coloneqq c_i(T_X)$.  For compact K\"ahler manifolds our local definition of a
holomorphic projective connection coincides with the normalized Atiyah-class
formulation of~\cite[Definition~1.12]{JahnkeRadloff2015}.  The general Chern-class
relation~\cite[Section~1.15, eq.~(1.16)]{JahnkeRadloff2015} specialises in
dimension $n$ to
\begin{equation}\label{eq:chern-id}
 2(n+1)\,c_2(X) = n\,c_1(X)^2 \qquad\text{in } H^4(X,\mathbb Q).
\end{equation}

Now let $X_d\subset\Pj^{n+1}$ be a smooth complex hypersurface of degree $d$ and
dimension $n\ge2$, and write $H\coloneqq c_1(\mathcal O_{X_d}(1))$ for the hyperplane class.
A standard computation  gives
\[
 2(n+1)\,c_2(X_d) - n\,c_1(X_d)^2
 = (n+2)(d-1)^2 H^2 .
\]
For $n\ge2$, the class $H^2$ is non-zero in $H^4(X_d,\mathbb Q)$; indeed,
$H^{n-2}\cdot H^2 = H^n = d \neq 0$.  Therefore the Chern-class identity
\eqref{eq:chern-id} forces $(n+2)(d-1)^2 = 0$, i.e.\ $d=1$.

Thus, among smooth projective hypersurfaces of dimension at least two,
Wronskian curvature positivity occurs only for the linear hypersurface $\Pj^n$.

\bigskip
\noindent{\bf Acknowledgements.}
The authors are grateful to  Junjiro Noguchi for his inspiring talk
``On Green--Griffiths--Lang Conjecture and Wronskian Curvature'' at the
Seminar on Complex Geometry, Institute of Mathematics, Academia Sinica,
Taipei, on 5~June~2026, in which he explained his second main theorem and
the idea of applying it to the Green--Griffiths--Lang conjecture.  This talk provided the initial motivation for the
present work.  We also warmly thank  Julie Tzu-Yueh Wang for
organizing the seminar.  We thank the Rethlas system for checking the correctness
of this paper.  We also used ChatGPT~5.5~Plus to help organize the manuscript
and assist with some of the computations.  S.-Y.\ Xie gratefully acknowledges
the hospitality of Academia Sinica during his visit.

\bigskip
\noindent{\bf Funding} \  
S.-Y. Xie acknowledges partial support from the National Key R\&D Program of China under Grants No. 2023YFA1010500 and No. 2021YFA1003100, and from the National Natural Science Foundation of China under Grants No. 12288201 and No. 12471081, as well as support from the  Xiaomi Young Talents Program.

\end{document}